\documentclass[12pt, leqno,twoside]{article}
\usepackage[normalem]{ulem}
\usepackage{amsmath,amsfonts,amssymb,amsthm}
\usepackage{enumerate}
\usepackage{amssymb}
\usepackage{amsmath}
\usepackage{amsthm}
\usepackage{graphicx}
\usepackage[pagebackref,colorlinks,citecolor=blue,linkcolor=blue]{hyperref}
\usepackage{cite}

\usepackage{amsmath}
\usepackage{amsfonts}
\usepackage{amssymb}
\usepackage{amsmath,bbm,amssymb,amsxtra}
\usepackage{mathrsfs}
\usepackage{enumerate}
\usepackage{caption}
\allowdisplaybreaks
\usepackage{epsfig}
\usepackage{ae}
  \def\Xint#1{\mathchoice
    {\XXint\displaystyle\textstyle{#1}}%
    {\XXint\textstyle\scriptstyle{#1}}%
    {\XXint\scriptstyle\scriptscriptstyle{#1}}%
    {\XXint\scriptscriptstyle\scriptscriptstyle{#1}}%
    \!\int}

    \def\XXint#1#2#3{{\setbox0=\hbox{$#1{#2#3}{\int}$ }
    \vcenter{\hbox{$#2#3$ }}\kern-.6\wd0}}
    \def\dashint{\Xint-}

\newcommand{\Id}{{\rm Id}\,}

\newtheorem{thm}{Theorem}[section]
\newtheorem{lem}[thm]{Lemma}

\newtheorem{rem}[thm]{Remark}
\newtheorem{cor}[thm]{Corollary}

\newtheorem{example}[thm]{Example}
\numberwithin{equation}{section}

\begin{document}
	\title{\Large\bf  Classification criteria for regular trees
		\footnotetext{\hspace{-0.35cm}
			$2010$ {\it Mathematics Subject classfication}: 31C05, 31C15, 31C45, 31E05
			\endgraf{{\it Key words and phases}:  capacity, harmonic function, parabolicity, regular tree}
			\endgraf{The author has been supported by the Academy of Finland (project No. 323960)
}
	}
	}
	\author{Khanh Nguyen \thanks{Department of Mathematics and Statistics, University of Jyv\"askyl\"a, PO~Box~35, FI-40014 Jyv\"askyl\"a, Finland. E-mail address: khanh.n.nguyen@jyu.fi}}
	
	\date{}
	\maketitle
	\begin{abstract}
		We give  characterizations for the parabolicity of regular trees.
	\end{abstract}
	
\section{Introduction} 
\ \ \ \  Let us begin with the uniformization theorem of F. Klein, P. Koebe and H. Poincar\'e for Riemann surfaces. The celebrated theorem says that every simply connected Riemann surface $M$ is conformally equivalent (or bi-holomorphic) to one of three Riemann surfaces: the half plane $\mathbb{H}^2$ (surface of hyperbolic type), the complex plane $\mathbb{R}^2$ (surface of parabolic type),  the Riemann sphere $\mathbb{S}$ (surface of elliptic type). Then $M$ admits a Riemannian metric $g$ with constant curvature.
A simply connected Riemann surface is said to be \textit{hyperbolic} if it is conformally equivalent to $\mathbb{H}^2$, otherwise we say that it is \textit{parabolic}.

Let $M$ be a simply connected Riemann surface with Riemannian metric $g$. A $C^2$-smooth function $u$ defined in $M$ is \textit{superharmonic} if 
\[-\Delta u\geq 0
\]where $\Delta$ is the Laplace-Beltrami operator associated to  the Riemannian metric $g$.

It is well known that every conformal mapping in dimension two preserves superharmonic functions (see \cite[Page 135]{AS60}). Since $\mathbb{H}^2$ possesses a nonconstant nonnegative superharmonic function and every nonnegative superharmonic function on $\mathbb{R}^2$ or $\mathbb{S}$ is constant, it then follows that there is no  nonconstant nonnegative superharmonic function on $(M,g)$ if and only if $M$ is parabolic.

Let $K$ be a compact  subset in $(M,g)$. We define the \textit{capacity} $\text{\rm Cap}(K)$ by 
\[\text{\rm Cap}(K)=\inf \left \{\int_M |\nabla u|^2 dm_g: u \in \text{\rm Lip}_0(M), u|_{K}\equiv 1\right \}
\]where  $\text{\rm Lip}_0(M)$ is a set of all Lipschitz functions with  compact support on $M$, and $m_g$ is the Riemannian measure associated to $g$. 
Then there is a nonconstant nonnegative superharmonic function on $M$ if and only if $\textit{\rm Cap}(K)>0$ for some compact subset $K$, (see \cite[Theorem 5.1]{G99} for Riemannian manifolds). It follows that the parabolicity of a Riemann surface $M$ can be characterized both in terms of capacity and superharmonic functions. By this reason, in the setting of Riemannian manifolds or metric measure spaces, one defines  parabolicity either via capacity (see \cite{H90,H3,Pekka01,HMP}) or via superharmonic functions (see \cite{G99} and also references therein). In this paper, we will consider  $K$-regular trees and give the definition of parabolicity  in terms of capacity.

	Recently, analysis on $K$-regular trees has been under development, see \cite{BBGS, PKW, PW19, KZ,pekkakhanhzhuang,Wang}. Let $G$ be a  $K$-regular  tree with  a set  of vertices $V$ and a set  of edges $E$ for some $K\geq 1$. The union of $V$ and $E$ will be  denoted  by $X$. We abuse the notation and call $X$ a $K$-regular tree. We introduce a metric structure on $X$ by considering each edge of $X$ to be an isometric copy of the unit interval. Then the distance between two vertices is the number of edges needed to connect them and there is a unique geodesic that minimizes this number.
	Let us denote the root by $0$. If $x$ is a vertex, we define $|x|$ to be the distance between $0$ and $x$. Since each edge is an isometric copy of the unit interval, we may extend this distance naturally to any $x$ belonging to an edge.
	We define $\partial X$ as the collection of all infinite geodesics starting at the root $0$. Then every $\xi\in \partial X$ corresponds to an infinite geodesic $[0,\xi)$ (in $X$) that is an isometric copy of the interval $[0,\infty)$.
	Let $\mu$ and $\lambda: [0,\infty)\to (0,\infty)$ be  locally integrable functions. Let $d|x|$ be the length element on $X$. We define a measure $\mu$ on $X$ by setting $d\mu(x)=\mu(|x|)d|x|$, and a metric $d$ on $X$ via $ds(x)=\lambda(|x|)d|x|$ by setting $d(y,z)=\int_{[y,z]}ds$ whenever $y, z\in X$ and $[y,z]$ is the unique geodesic between $y$ and $z$.
	Then $(X, d, \mu )$ is a metric measure space and hence one may define a Newtonian Sobolev space $N^{1,p}(X):=N^{1,p}(X,d,\mu)$ based on upper gradients \cite{HK98, N00}. As usual, $N^{1,p}_0(X) $ is the completion of the family of functions with compact support 
	in $N^{1,p}(X)$, and $\dot N^{1,p}_0(X)$ is the completion of the family of functions with compact support in $\dot N^{1,p}(X)$, the homogeneous version of $N^{1,p}(X)$. Let $\Omega$ be a subset of $X$. We denote by $N^{1,p}_{\rm loc}(\Omega)$ the space of all functions $u\in L^p_{\rm loc}(\Omega)$ that have an upper gradient in $L^p_{\rm loc}(\Omega)$, where $L^p_{\rm loc}(\Omega)$ is the space of all measurable functions that are $p$-integrable on any compact subset of $\Omega$.
	See Section \ref{Sec2} for the precise definitions.
	
	Let $1<p<\infty$ and $O$ be a subset of $X$. We define the \textit{$p$-capacity} of $O$, denoted $\text{\rm Cap}_p(O)$, by setting
	\begin{equation}
	\label{def-cap}\text{\rm Cap}_p(O)=\inf\left \{\int_X g_u^pd\mu: u|_{O}\equiv 1, u\in N^{1,p}_0(X)\right \}
	\end{equation}
	where $g_u$ is the minimal upper gradient of $u$ as in Section \ref{sec22}. A $K$-regular tree $X$ is said to be \textit{$p$-parabolic} if $\text{\rm Cap}_p(O)=0$ for all compact sets $O\subset X$; otherwise $X$ is \textit{$p$-hyperbolic}.

	Given $1<p<\infty$ and an open subset $\Omega \subseteq X$, we say that $u\in N_{\text{\rm loc}}^{1,p}(\Omega)$ is a \textit{$p$-harmonic function} (or a \textit{$p$-superharmonic function}) on $\Omega$ if 
		\begin{equation}
		\label{suph}\int_{\text{\rm spt}(\varphi)}g_u^p d\mu \leq \int_{\text{\rm spt}(\varphi)} g_{u+\varphi}^p d\mu
		\end{equation}
	 holds for all functions (or for all nonnegative functions) $\varphi\in N^{1,p}(\Omega)$ with compact support  $\text{\rm spt}(\varphi)\subset \Omega$. We refer the interested readers to \cite{BB11, tero, HKST15} for a discussion on the $p$-capacity and $p$-(super)harmonic functions.
		
		 Since a $K$-regular tree $(X,d)$ is the quintessential Gromov hyperbolic space, it is then 
		natural to ask for whether the parabolicity (or hyperbolicity) of $X$ can be characterized via $p$-(super)harmonic functions under some conditions on the measure $\mu$ only depending on the given metric $d$, and also ask for intrinsic conditions of $K$-regular trees that would characterize the parabolicity (or hyperbolicity). We refer  the readers to \cite[Chapter IV]{AS60}
		for a discussion in the case of Riemann surfaces, and \cite {GM99, G99, H3, Pekka01, HMP} for a discussion in the setting of Riemannian manifolds, and \cite[Section 6]{T99}, \cite{Y77}  for a discussion on infinite networks.

	In order to state our results, we introduce a notion from  \cite{PKW}. Let $1<p<\infty$. We set 
		\[R_p(\lambda,\mu)=\int_0^{\infty} \lambda(t)^{\frac{p}{p-1}}\mu(t)^{\frac{1}{1-p}}K^{\frac{ j(t) }{1-p}} dt\]
where $j(t)$ is the smallest integer such that $j(t)\geq t$, and let $X^n=\{x\in X: |x|\leq n\}$ for each $n\in \mathbb{N}$. Since we work with a fixed pair $\lambda, \mu$, we will usually write $R_p(\lambda,\mu)$ simply as $R_p$ when no confusion can arise. In what follows, we additionally assume that $\lambda^p\mu^{-1}\in L^{{1}/(p-1)}_{\rm loc}([0,\infty))$ to make sure that the finiteness of $R_p$ is a condition at infinity.

The first  result of our paper is a characterization of parabolicity of $K$-regular trees.
	\begin{thm}\label{main theorem}Let $1<p<\infty$ and $X$ be a $K$-regular tree with metric $d$ and measure $\mu$ as above, {with $K\geq 1$}. Then $(X,d,\mu )$ is p-parabolic if and only if any one of the following conditions is fulfilled:
		\begin{enumerate}
			\item $R_p(\lambda,\mu)=\infty$.
			\item $\text{\rm Cap}_p(X^n)=0$ for all $n\in \mathbb{N}\cup \{0\}$.
			\item $\text{\rm Cap}_p(X^n)=0$ for some $n\in \mathbb{N}\cup \{0\}$.
		\end{enumerate}
	\end{thm}
 In Section \ref{sec21}, we will show that the compactness on a $K$-regular tree $X$ with respect to our metric $d$ and with respect to the graph metric are equivalent. Since each compact set in $(X,d)$ is contained in some $n$-level set $X^n$ that is an analog of a ball with respect to graph metric,  parabolicity of $X$ can be characterized by the zero $p$-capacity of some/ all $n$-level sets $X^n$.
 
 In \cite[Theorem 1.3]{PKW}, the condition $R_p(\lambda,\mu)=\infty$ gives a characterization of the existence of boundary  trace operators and for density properties for $\dot N^{1,p}(X)$. Hence parabolicity of $K$-regular trees can be characterized in terms of boundary trace operators and density properties. Combining Theorem \ref{main theorem} and \cite[Theorem 1.3 and Theorem 3.5]{PKW}, we obtain the following corollary.
\begin{cor} Let $1<p<\infty$ and $X$ be a $K$-regular tree with metric $d$ and measure $\mu$ as above, {with $K\geq 1$}. Then $(X,d,\mu)$ is $p$-parabolic if and only if any one of the following conditions is fulfilled:
	\begin{enumerate}
		\item There exists $u\in \dot N^{1,p}(X)$ such that
		\[\lim_{[0,\xi)\ni x\to \xi }u(x)=\infty\]
		 for all $\xi \in \partial X$.
		\item $\dot N^{1,p}_0(X)=\dot N^{1,p}(X)$.
	\end{enumerate}
	\end{cor}

It is well known, see for instance the survey paper \cite{Pekka01}, that the volume growth condition
\[\int_1^\infty \left (\frac{t}{V(B(0,t))}\right )^{\frac{1}{p-1}}dt=\infty
\]is a sufficient condition to guarantee parabolicity of Riemannian manifolds. Here $V(B(0,t))$ is the volume of the ball with radius $t$ and center at a fixed point $0$. However, this condition is far from being necessary in general, as shown by a counterexample due to I. Holopainen  \cite{H3} 
and to Varopoulos\cite{Varo} in the case $p=2$.
Our condition $R_p(\lambda,\mu)=\infty$ is an analog of this volume growth  condition. Example \ref{example} in Section \ref{Sec3} shows that there exists a $K$-regular tree with a distance $d$ and  a ``{non-radial}''  measure $\mu$  such that $R_p(\lambda,\mu)=\infty$ but $X$ is $p$-hyperbolic.

Let $1<p<\infty$ and let $\text{\rm int}(X^n):=\{x\in X:|x|<n\}$ for $n\in\mathbb N$. We  say that $(\text{\rm int}(X^n), d, \mu)$ is \textit{doubling} and supports \textit{a $p$-Poincar\'e inequality} if there exist  constants $C_1\geq 1, C_2>0$ only depending on $n$  such  that for all  balls  $B(x,2r)\subset \text{\rm int}(X^n)$, 
\[\mu\left (B(x,2r)\right )\leq C_1\mu\left (B(x,r)\right )\] 
and for all balls $B(x, r)\subset \text{\rm int}(X^n)$, \[\dashint_{B(x,r)}|u-u_{B(x,r)}|d\mu\leq C_2r\left (\dashint_{B(x, r)}g^pd\mu\right )^{\frac{1}{p}}\]whenever $u$ is a measurable function on $B(x, r)$ and $g$ is an upper gradient of $u$, where  $u_{B(x,r)}:=\dashint_{B(x,r)}ud\mu=\frac{1}{\mu(B(x,r))}\int_{B(x,r)}ud\mu.$ The validity of a $p$-Poincar\'e inequality for $X$ has very recently been characterized via a Muckenhoupt-type condition under a doubling condition on $(X,d,\mu)$, see \cite{KZ} for more information.

Our second result deals with a characterization of parabolicity in terms of $p$-(super)- harmonic functions.
\begin{thm}\label{main theorem2}
	Let $1<p<\infty$ and $X$ be a $K$-regular tree with metric $ d$ and measure $\mu$ as above, with $K\geq 2$. Assume additionally that $(\text{\rm int}(X^n), d, \mu)$ is doubling and supports a $p$-Poincar\'e inequality for each $n\in \mathbb{N}$. Then $(X, d, \mu)$ is p-parabolic if and only if any one of the following conditions is fulfilled:
	\begin{enumerate}
					\item Every nonnegative $p$-superharmonic function $u$ on $X$ is constant.
					\item Every nonnegative $p$-harmonic function $u$ on $X$ is constant.
					\item Every bounded $p$-harmonic function $u$ on $X$ is constant.
					\item Every bounded $p$-harmonic function $u$ on $X$ with  $\int_Xg_u^pd\mu<\infty$ is constant.
	\end{enumerate}
\end{thm}
Let us close the introduction with some comments on Theorem \ref{main theorem2}. According to a version of Theorem \ref{main theorem2} in  the setting of Riemannian manifolds from \cite{H90, H91, H93} we have that $1.\Rightarrow 2.\Rightarrow 3.\Rightarrow 4.$ However $3.$ does not imply $2.$ in general. 
Our condition that $(\text{\rm int}(X^n),d,\mu)$ is  doubling and supports a $p$-Poincar\'e inequality for each $n\in\mathbb{N}$ is equivalent to $\mu$ being a locally doubling measure supporting a local $p$-Poincar\'e inequality on $(X,d)$.

Theorem \ref{main theorem2} is not empty in the sense that there exist both $p$-parabolic and $p$-hyperbolic $K$-regular trees that are doubling and support a $p$-Poincar\'e inequality,
see Example \ref{example2} in Section \ref{Sec3} for more details.

The motivation for our paper comes from classification problems of spaces.
By the survey papers \cite{note, G99}, the development of potential theory in the setting of metric measure spaces leads to a classification of spaces as either $p$-parabolic or not. This dichotomy can be seen as a non-linear analog of the recurrence or transience dichotomy in the theory of Brownian motion. This classification is helpful in the development of a quasiconformal uniformization theory, or for a deeper understanding of the links between the geometry of hyperbolic spaces and the analysis on their boundaries at infinity.

The paper is organized as follows. In Section \ref{Sec2}, we introduce  $K$-regular trees, Newtonian spaces, and $p$-(super)harmonic functions on our trees. In Section \ref{Sec3}, we give the  proofs of Theorem \ref{main theorem} and Theorem \ref{main theorem2}.

	Throughout  this paper, the letter $C$ (sometimes  with a subscript)  will denote positive constants that usually depend only on the  space and may change at  different occurrences; if $C$  depends on $a, b, \ldots$, we  write $C=C(a,b, \ldots)$. For  any function $f\in L^1_{\rm loc}(X)$ and any measurable subset $A\subset X$,  let  $ \dashint_Afd\mu$ stand for  $\frac{1}{\mu(A)}\int_Afd\mu$.
	\section{Preliminaries}\label{Sec2}

	\subsection{Regular trees}\label{sec21}
\ \ \ \ 	{A {\it graph} $G$ is a pair $(V, E)$, where $V$ is a set of vertices and $E$ is a set of edges.   We call a pair of vertices $x, y\in V$  neighbors if $x$ is connected to $y$ by an edge. The degree of a vertex is the number of its neighbors. The graph structure gives rise to a natural connectivity structure. A {\it tree} $G$ is a connected graph without cycles. A graph (or tree) is made into a metric graph by considering each edge as a geodesic of length one.
		
		We call a tree $G$ a {\it rooted tree} if it has a distinguished vertex called the {\it root}, which we will denote by $0$. The neighbors of a vertex $x\in V$ are of two types: the neighbors that are closer to the root are called {\it parents} of $x$ and all other  neighbors  are called {\it children} of $x$. Each vertex has a unique parent, except for the root itself that has none. 
		
		We say that a tree is  {\it $K$-regular} if it is a rooted tree such that each vertex has exactly $K$ children for some integer $K\geq 1$.
		Then all vertices except the root  of  a $K$-regular tree have degree $K+1$, and the root has degree $K$. 
		
		Let $G$ be a $K$-regular tree with a set of vertices $V$ and a set of edges $E$ for some integer $K\geq 1$. For simplicity of notation, we let $X=V\cup E$ and call it a $K$-regular tree.
		For $x\in X$, let $|x|$ be the distance from the root $0$ to $x$, that is, the length of the geodesic from $0$ to $x$, where the length of every edge is $1$ and we consider each edge to be an isometric copy of the unit interval. The geodesic connecting two points $x, y\in X$ is denoted by $[x, y]$. Throughout this paper, we denote $X^n:=\{x\in X: |x|\leq n\}$ and $\textit{\rm int}(X^n):=\{x\in X:|x|<n\}$ for each $n\in\mathbb N$.
		
		On our $K$-regular tree $X$, we define a measure $\mu$ and a metric $d$ via $ds$ by setting
		\begin{equation*}
		d\mu(x)=\mu(|x|)\,d|x|,\ \  ds(x)=\lambda(|x|)\, d|x|,
		\end{equation*}
		where $\lambda, \mu:[0, \infty)\rightarrow (0, \infty)$ are fixed with $\lambda, \mu\in L^1_{\rm loc}([0, \infty))$.   
		Here $d\,|x|$ is the measure which gives each edge Lebesgue measure $1$, as we consider each edge to be an isometric copy of the unit interval and the vertices are the end points of this interval. Hence for any two points $z, y\in X$, the distance between them is 
		\[d(z, y)=\int_{[z, y]} \,ds(x)=\int_{[z, y]}\lambda(|x|)\, d|x|\]
		where $[z, y]$ is the unique geodesic from $z$ to $y$ in $X$.
		
		We abuse the notation and let $\mu(x)$ and $\lambda(x)$ denote $\mu(|x|)$ and $\lambda(|x|)$, respectively, for any  $x\in X$, if there is no danger of confusion.
		
		We denote by $d_E$ the graph metric on $X$. Then for any two points $z,y\in X$,
		\[d_E(z,y)=\int_{[z,y]}d|x|
				\] is  the graph distance between $z$ and $y$ where $[z,y]$ is the unique geodesic from $z$ to $y$.
		\begin{thm}\label{lemma-topo} The identity mapping $\Id_X: (X, d_E)\to (X, d)$ is a homeomorphism. 
		\end{thm}
		\begin{proof} Let us first prove that the identity mapping  $f:(X,d_E)\to (X,d)$, $f(x)=x$ if $x\in X$, is continuous.
			Let $B_d(x,r)$ be an arbitrary open ball with center $x$ and radius $r>0$ in $(X,d)$. Recall that $\lambda: [0, \infty)\to (0, \infty)$ is a locally integrable function. Hence  $\lambda$ is an integrable function on $[a,b]$ whenever $[a,b]$ is a compact interval with  $|x|\in (a,b)$ if $x\neq 0$, or $|x|=a$ if $x=0$ where $0$ is the root of $X$. Then
			\[F(h):=\int_{a}^h\lambda(t)dt
			\]
			is absolutely continuous on $[a, b]$. It follows that there exists $\delta_r>0$ only depending on $x, r$ such that 
			\[
			\begin{cases}
			\int_{|x|-\delta_r}^{|x|+\delta_r} \lambda(t)dt <\frac{r}{2} & \text{ if }|x| \in (a, b), x\neq 0,\\ 
			\int_{|x|}^{|x|+\delta_r}\lambda(t)dt<\frac{r}{2} &\text{ if }|x|=0.
			\end{cases}
			\]
			The open ball with center $x$ and radius $\delta_r$ in $(X, d_E)$ is denoted by $B_{d_E}(x,\delta_r)$. For any $y\in B_{d_E}(x,\delta_r)$, we have that $[x,y]\subset [x,\bar x]\cup[\bar x,y]$ where $\bar x\in[0,x]$ with $d_E(x,\bar x)=\delta_r$. Then the above estimate gives that 
			\[\begin{cases}
			d(x,y)=\int_{[x,y]}\lambda(t)dt<2\int_{|x|-\delta_r}^{|x|+\delta_r}\lambda(t)dt <r& \text{ if }|x|\in (a,b), x\neq 0,\\
			d(x,y)=\int_{[x,y]}\lambda(t)dt<2\int_{|x|}^{|x|+\delta_r}<r &\text{ if }|x|=0,
			\end{cases}
			\]and hence $y\in B_d(x,r)$. As $B_d(x,r)$ is arbitrary, we obtain that for any open ball $B_d(x,r)$ there exists $\delta_r>0$ only depending on $x, r$ such that	  $B_{d_E}(x,\delta_r)\subset B_d(x,r).$	Thus
			\begin{equation}\label{fcontinuous}
			\text{the identity mapping } f: (X,d_E)\to (X,d) \text{ is continuous}.
			\end{equation}
			Next, we claim that also the identity mapping $g: (X,d)\to (X,d_E)$, $g(x)=x$ if $x\in X$, is continuous. Let $B_{d_E}(x,r')$ be an arbitrary open ball with center $x$ and radius $r'>0$  in $(X,d_E)$.  We set
			\begin{equation}
			\label{delta-r'}\delta_{r'}= \min \left \{\int_{|x|-r'/3}^{|x|}\lambda(t)dt , \int_{|x|}^{|x|+r'/3}\lambda(t)dt \right \}.
			\end{equation}
			Then $\delta_{r'}>0$ since $\lambda>0$.
			We denote by $B_d(x,\delta_{r'})$ the open ball with  center $x$ and radius $\delta_{r'}$ in $(X,d)$.
			For any $y\in B_d(x,\delta_{r'})$, we have that
\begin{equation}
\label{delta-r'1}\int_{[x,y]}\lambda(t)dt=d(x,y)<\delta_{r'}.
\end{equation}	
			It follows from \eqref{delta-r'} and \eqref{delta-r'1} that $|z|\in [|x|-r'/3, |x|+r'/3]$ for any $z\in [x,y]$, and hence $d_E(x,z)<r'$ for any $z\in[x,y]$. In particular, $d_E(x,y)<r'$ for any $y\in B_d(x,\delta_{r'})$.
			Then 
			$B_{d}(x,\delta_{r'})\subset B_{d_E}(x,r')
			$ for any $B_{d_E}(x,r')$. Therefore
			\begin{equation}\label{gcontinuous}
			\text{the identity mapping }g:(X,d)\to (X,d_E) \text{ is continuous.}
			\end{equation}
			We conclude from \eqref{fcontinuous} and \eqref{gcontinuous} that $\Id_X:(X,d_E)\to (X,d)$ is a  homeomorphism.
			The claim follows.
		\end{proof}
	We note that $X^n$ is compact in $(X,d_E)$ for each $n\in\mathbb{N}$ because it is a union of finitely many compact edges. Furthermore, any compact set in $(X,d_E)$ is contained in $X^n$ for some $n$ since any compact set in $(X,d_E)$ is bounded.  Since compactness is preserved under homeomorphisms, we have the following corollaries.
	\begin{cor}\label{cor2.2}
		Let $O$ be an arbitrary compact set in $(X,d)$. Then $O\subset X^n$ for some $n\in\mathbb{N}$.
	\end{cor}
\begin{cor}\label{cor2.3}
	Let $n\in \mathbb{N}$. Then $X^n$ is compact in $(X,d)$, and $\text{\rm int}(X^n)$ is open in $(X,d)$.
\end{cor}
\begin{cor}\label{cor2.4}
	$(X, d, \mu)$ is a connected, locally compact, and non-compact metric measure space.
\end{cor}
		\subsection{Newtonian spaces}\label{sec22}
\ \ \ \		Let $1<p<\infty$ and $X$ be a $K$-regular tree with metric $d$ and measure $\mu$ as in Section \ref{sec21}. Let $u\in L^{1}_{\text{\rm loc}}(X)$. We say that a Borel function $g:X\to[0,\infty]$ is an \textit{upper gradient} of $u$ if 
		\begin{equation}\label{2.1}|u(y)-u(z)|\leq \int_\gamma g ds
		\end{equation}
		whenever $y, z\in X$ and $\gamma$ is the geodesic from $y$ to $z$.  In the setting of our tree, any rectifiable curve with end points $z$ and $y$ contains the geodesic connecting $z$ and $y$, and therefore the upper gradient defined above is equivalent to the definition which requires that \eqref{2.1} holds for all rectifiable curves with end points $z$ and $y$. 
		In \cite{H03,HKST15}, the notion of a $p$-weak upper gradient is given. A Borel function $g: X\rightarrow [0, \infty]$ is called a $p$-weak upper gradient of $u$ if \eqref{2.1} holds on $p$-a.e. curve. Here we say that a property holds for {\it $p$-a.e. curve} if it fails only for a  curve family $\Gamma$ with {\it zero $p$-modulus}, i.e., there is a Borel nonnegative function $\rho\in L^p(X)$ such that $\int_{\gamma}\rho\, ds=\infty$ for any curve $\gamma\in \Gamma$.
		We refer to \cite{H03,HKST15} for more information about $p$-weak upper gradients.
		
		The notion of upper gradients is due to Heinonen and Koskela \cite{HK98}, we refer interested readers to \cite{BB11,H03,HKST15,N00} for a more detailed discussion on upper gradients.

		The following lemma of Fuglede shows that a converging sequence in $L^p$ has a subsequence that converges with respect to $p$-a.e curve (see \cite[Section 5.2]{HKST15}).
		\begin{lem}[Fuglede's lemma] \label{lemma21}
			Let $\{g_n\}_{n=1}^{\infty}$ be a sequence of Borel nonnegative functions that converges to $g$ in $L^p(X)$. Then there is a subsequence $\{g_{n_k}\}_{k=1}^{\infty}$ such that 
			\[\lim_{k\to\infty}\int_{\gamma}|g_{n_k}-g|ds=0
			\]for $p$-a.e curve $\gamma$ in $X$.
		\end{lem}
		The following useful results are from \cite[Section 2.3 and Section  2.4]{HKST15} or \cite[Section 6.1]{BB11}.
		\begin{thm}\label{theorem22}
			Every bounded sequence $\{u_n\}_{n=1}^{\infty}$ in a reflexive normed space $(V, |.|_V)$ has a weakly convergent subsequence $\{u_{n_k}\}_{k=1}^\infty$. Moreover, there exists $u\in  V$ such  that $u_{n_k}\to u$ weakly in $V$ as $k\to\infty$ and 
			\[|u|_V\leq \liminf_{k\to \infty}|u_{n_k}|_V.
			\]
		\end{thm}
		\begin{lem}[{Mazur's lemma}] \label{lemma23}
			Let $\{u_n\}_{n=1}^\infty$ be a sequence in a normed space $V$ converging weakly to an element $u\in V$. Then there exists a sequence $\bar v_k$ of convex combinations
			\[\bar v_k=\sum_{i=k}^{N_k}\lambda_{i,k}u_i \ , \sum_{i=k}^{N_k}\lambda_{i,k}=1\ , \lambda_{i,k}\geq 0
			\]
			converging to $v$ in the norm.			
			\end{lem}
		The {\it Newtonian space} $N^{1, p}(X)$, $1< p<\infty$, is defined as the collection of all the functions $u$ with finite $N^{1, p}$-norm 
		$$\|u\|_{N^{1, p}(X)}:= \|u\|_{L^p(X)}+\inf_g \|g\|_{L^p(X)}$$
		where the infimum is taken over all upper gradients of $u$. We denote by $g_u$ the minimal upper gradient, which is unique up to measure zero and which is minimal in the sense that if $g\in L^p(X)$ is any upper gradient of $u$ then $g_u\leq g$ a.e.. We refer to \cite[Theorem 7.16]{H03} for proofs of the existence and uniqueness of such a minimal upper gradient. 
		
		If $u\in N^{1, p}(X)$, then it is continuous by \eqref{2.1} under the assumption $\frac{\lambda^p}{\mu}\in L^{\frac{1}{p-1}}_{\rm loc}([0,\infty))$ and  it has a minimal $p$-weak upper gradient, see \cite[Section 2]{PKW}. More precisely, by \cite[Proposition 2.2]{PKW} the empty family is
		the only curve family with zero $p$-modulus, and hence any $p$-weak upper gradient is actually an upper gradient here and the conclusion of Lemma \ref{lemma21} holds for every curve $\gamma$. Moreover, it follows from \cite[Definition 7.2 and Lemma 7.6]{H03} that any function $u\in L^1_{\rm loc}(X)$ with an upper gradient $0\leq g\in L^p(X)$ is locally absolutely continuous, for example, absolutely continuous on each edge. The “classical” derivative $u'$ of this locally absolutely continuous function is a minimal upper gradient in the sense that $g_u = |u'(x)|/\lambda(x)$ when $u$ is parametrized in the natural way.

		We define the {\it homogeneous Newtonian space} $\dot{N}^{1, p}(X)$, $1< p<\infty$, the collection of all the continuous functions $u$ that have an
		upper gradient $0\leq g\in L^p(X)$, for which the homogeneous $\dot{N}^{1, p}$-norm of $u$ defined as
		$$\|u\|_{\dot N^{1, p}(X)}:= |u(0)|+\inf_g  \|g\|_{L^p(X)}$$
		is finite. Here $0$ is the root of our $K$-regular tree $X$ and the infimum is taken over all upper gradients of $u$.
		
		 The  completion of the family of functions with compact support in $N^{1,p}(X)$ (or  $\dot N^{1,p}(X)$) is denoted by $N^{1,p}_0(X)$ (or $\dot N^{1,p}_0(X)$). We denote by $N^{1,p}_{\rm loc}(X)$ the space of all functions $u\in L^p_{\rm loc}(X)$ that have an upper gradient in $L^p_{\rm loc}(X)$, where $L^p_{\rm loc}(X)$ is the space of all measurable functions that are $p$-integrable on any compact subset of $X$. Especially, since each $X^n$ is compact in $(X,d)$ by Corollary \ref{cor2.3}, we conclude that each $u\in N^{1,p}_{\rm loc}(X)$ is both continuous and bounded on each $X^n$.
		\subsection{$p$-(super)harmonic functions}\label{sec2.3}
\ \ \ \			Let $1<p<\infty$ and  $X$ be a $K$-regular tree with metric $d$ and measure $\mu$ as  in Section \ref{sec21}. For an open subset $\Omega$ of  $ X$,  a function $u\in N^{1,p}_{\rm loc}(\Omega)$ is said to be a \textit{$p$-harmonic function} on $\Omega$ if \begin{equation}
			\label{1.sec23}\int_{\text{\rm spt}(\varphi)}g_u^pd\mu\leq \int_{\text{\rm spt}(\varphi)}g_{u+\varphi}^pd\mu 
			\end{equation}holds for all functions $\varphi\in N^{1,p}(\Omega)$ with compact support $\text{\rm spt}(\varphi)\subset \Omega$. We say that a function $u\in N^{1,p}_{\rm loc}(\Omega)$ is a \textit{$p$-superharmonic function} if \eqref{1.sec23} holds for all nonnegative functions $\varphi\in N^{1,p}(\Omega)$ with compact support $\text{\rm spt}(\varphi)\subset \Omega$.
		
We  have a characterization of $p$-harmonic functions on $X$, see \cite[Lemma 7.11]{BB11}.
			\begin{thm}\label{theorem24}
			A function $u$ is a $p$-harmonic function  on $X$ if and only if $u$ is  $p$-harmonic  on $\text{\rm int}(X^n)$ for all $n\in \mathbb{N}$.
		\end{thm}		
		By  the stability properties of $p$-superharmonic functions (superminimizers) in general metric measure  spaces (see for instance \cite[Theorem 7.25]{BB11}),  since $\text{\rm int}(X^n)$ is open for each $n\in \mathbb{N}$ (see Corollary \ref{cor2.3}), we obtain the following results in our setting.
	
		\begin{thm}\label{theorem25}
			 Let $n\in\mathbb N$. If $\{u_i\}_{i\geq n}$ is a sequence of $p$-harmonic functions on $\text{\rm int}(X^n)$ which converges locally uniformly to $u$ in $\text{\rm int}(X^n)$, then $u$ is $p$-harmonic on $\text{\rm int}(X^n)$.
		\end{thm}
		Let $1<p<\infty$ and $n\in\mathbb N$. Then  $(\text{\rm int}(X^n), d, \mu)$ is said to be \textit{doubling} and to support a \textit{$p$-Poincar\'e inequality}  if  there exist constants $C_1\geq 1, C_2>0$ only depending on $n$ such  that 		
		for all  balls  $B(x,2r)\subset \text{\rm int}(X^n)$, 
		\[\mu(B(x,2r))\leq C_1\mu(B(x,r))\] 
		and for all balls $B(x, r)\subset \text{\rm int}(X^n)$, \[\dashint_{B(x,r)}|u-u_{B(x,r)}|d\mu\leq C_2r\left (\dashint_{B(x, r)}g^pd\mu\right )^{\frac{1}{p}}\]whenever $u$ is a measurable function on $B(x, r)$ and $g$ is an upper gradient of $u$.
		Recall that if $(\text{\rm int}(X^n), d, \mu)$ is doubling and supports a $p$-Poincar\'e inequality then $N^{1,p}(\text{\rm int}(X^n), d, \mu)$ is a reflexive space  (see \cite[Theorem 4.48]{C99}). 
	
	Combining Proposition 3.9 and Theorem 5.4 in \cite{KS01}, we obtain the local H\"older continuity of $p$-harmonic functions on $\text{\rm int}(X^n)$ for each $n\in\mathbb{N}$.
		\begin{thm}\label{theorem26}
			Let $n\in \mathbb{N}$. Assume that $(\text{\rm int}(X^n), d, \mu)$ is doubling and supports a $p$-Poincar\'e inequality. Then every $p$-harmonic function $u$ on $\text{\rm int}(X^n)$ is locally $\alpha$-H\"older continuous  for some $0<\alpha\leq 1$.
		\end{thm}
	\section{Proofs of Theorem \ref{main theorem} and Theorem \ref{main theorem2}}  \label{Sec3}
\ \ \ \ \ In this section, if we do not specifically mention, we always assume that $1<p<\infty$ and that $X$ is a $K$-regular tree with metric $d$ and measure $\mu$ as in Section \ref{sec21}.
	\begin{lem} \label{lem1}
		$X$ is $p$-parabolic if and only if $\text{\rm Cap}_p(X^n)=0$ for all $n \in\mathbb{N}\cup \{0\}$.
		\end{lem}
	\begin{proof}
		 Let $X$ be $p$-parabolic. By Corollary \ref{cor2.3}, we have that $X^n$ is compact in $(X,d)$ for all $n\in\mathbb{N}$ and hence  \[\text{\rm Cap}_p(X^n)=0\] for all $n\in \mathbb{N}$. This also holds for all $n\in \mathbb{N}\cup \{0\}$ because $\text{\rm Cap}_p(X^0)\leq \text{\rm Cap}_p(X^n)$.
		Conversely, suppose that 
		\begin{equation}
		\label{cap-lem2}\text{\rm Cap}_p(X^n)=0
		\end{equation}
		 for all $n\in\mathbb{N}\cup\{0\}$. Let $O$ be an arbitrary compact set in $(X,d)$. Then $O\subset X^n$ for some $n\in \mathbb{N}$ by Corollary \ref{cor2.2}, and so that $\text{\rm Cap}_p(O)\leq \text{\rm Cap}_p(X^n)$. Combining this with \eqref{cap-lem2} yields  $\text{\rm Cap}_p(O)=0$. Since $O$ is arbitrary, we conclude that $X$ is $p$-parabolic. The proof is complete.
	\end{proof}
	\begin{lem}\label{prop}
		Let $n\in \mathbb{N}\cup\{0\}$ be arbitrary. Then $R_p=\infty$ if and only if $\text{\rm Cap}_p(X^n)=0$.
	\end{lem}
\begin{proof}Suppose that  $R_p=\infty$.  We first claim that $\text{\rm  Cap}_p(X^n)=0$.  
	Let us define a sequence 
	$\{u_k\}_{k=n+1}^{\infty}$ by  setting
	\begin{equation}
	\label{uk} u_k(x)=\begin{cases}
	1 & \text{ if } x\in X^n,\\
	1-\frac{\int_n^{|x|}\lambda^{\frac{p}{p-1}}(t)\mu^{\frac{1}{1-p}}(t)K^{\frac{j(t)}{1-p}}dt}{\int_n^{k}\lambda^{\frac{p}{p-1}}(t)\mu^{\frac{1}{1-p}}(t)K^{\frac{j(t)}{1-p}}dt}&\text{ if }x\in X^k\setminus X^n,\\
	0&\text{ otherwise }.
	\end{cases}
	\end{equation}
	Then
	\[g_k(x)=\frac{\lambda^{\frac{1}{p-1}}(x)\mu^{\frac{1}{1-p}}(x)K^{\frac{j(x)}{1-p}}}{\int_n^k\lambda^{\frac{p}{p-1}}(t)\mu^{\frac{1}{1-p}}(t)K^{\frac{j(t)}{1-p}}dt}\chi_{X^k\setminus X^n}(x)
	\] is an upper gradient of $u_k$. 
	  Next, a direct computation reveals that
	\begin{equation}
	\label{guk}\int_Xg_{u_k}^pd\mu \leq \int_Xg_k^pd\mu = \frac{1}{\left (\int_n^{k}\lambda^{\frac{p}{p-1}}(t)\mu^{\frac{1}{1-p}}(t)K^{\frac{j(t)}{1-p}}dt\right )^{p-1}}<\infty
	\end{equation}
	for all $k\geq n+1$. Since $R_p=\infty$ and $\mu(X^k)<\infty$ for each $k\in \mathbb{N}$, it follows from \eqref{uk}-\eqref{guk} that $u_k\in N^{1,p}_0(X)$ with $u_k|_{X^n}\equiv  1$ and that
	\[\lim_{k\to\infty}\int_Xg_{u_k}^pd\mu=0.
	\]We thus get $\text{\rm Cap}_p(X^n)=0$.
	Conversely, suppose that  $\text{\rm Cap}_p(X^n)=0$. Then there exists a sequence $\{u_k\}_{k=1}^\infty$ in $N^{1,p}_0(X)$ with $u_k|_{X^n}\equiv 1$ such that 
	\begin{equation}
	\label{gu_k}\lim_{k\to \infty}\int_{X}g_{u_k}^pd\mu=0.
	\end{equation}
	Set $v_k:=u_k-1$. Then $g_{u_k}=g_{v_k}$.
	Combining this with \eqref{gu_k} and $u_k(0)=1$ yields
	\[\| u_k-1\|^p_{\dot N^{1,p}(X)}=\| v_k\|^p_{\dot N^{1,p}(X)}=\int_{X}g_{v_k}^pd\mu  \to 0 , \text{ as }k\to \infty.
	\]
	Therefore $u_k\to 1$ in $\dot N^{1,p}(X)$ with $u_k\in  N^{1,p}_0(X)$, and hence  $1\in \dot N_0^{1,p}(X)$. Recall that $R_p=\infty$ is equivalent to $1\in \dot N^{1,p}_0(X)$ by \cite[Theorem 1.3 and Corollary 4.2]{PKW}. Thus $R_p=\infty$ which completes the proof.
\end{proof}	
\begin{lem}\label{lem2} 
	Let $X$ be  $p$-parabolic. Then every nonnegative $p$-superharmonic function $u$  on $X$ is constant.
\end{lem}

	\begin{proof}
		Let $u\in N^{1,p}_{\text{\rm loc}}(X)$ be an arbitrary  nonnegative $p$-superharmonic function on $X$. We claim that $u$ is  constant. Indeed,  let
 $n_0\in \mathbb{N}$ be arbitrary. We denote
\[M:=\| u\|_{L^\infty(X^{n_0})}.
\]
Then $M<\infty$, since  $u\in N^{1,p}_{\rm loc}(X)$ is bounded on $X^n$ for each $n\in \mathbb{N}$, see the  end of Section \ref{sec22}. By Lemma \ref{lem1}, we have $\text{\rm Cap}_p(X^{n_0})=0$, and hence that there is a sequence $\{1_n\}_{n=1}^{\infty}$ in $N^{1,p}_0(X)$ with $1_n|_{X^{n_0}}\equiv 1$ such that 
	\begin{equation}\label{0-thm1}\lim_{n\to \infty}\int_{X}g_{ 1_n}^pd\mu = 0 .
\end{equation}Without loss of generality we assume that each $\text{\rm spt}(1_n)$ is compact.  We define a sequence $\{\varphi_n\}_{n=1}^{\infty}$  by setting \[\varphi_n(x)=\max\{M \cdot 1_n(x), u(x) \}-u(x)\] for each $n\in \mathbb{N}$ and for all $x\in X$. Then 
\begin{equation}
\label{sptvarphi}\text{\rm spt}(\varphi_n)\subset \text{\rm spt}(1_n)
\end{equation}
for all $n\in \mathbb{N}$. We have that  $0\leq \varphi_n \in N^{1,p}(X)$ with  compact support $\text{\rm spt}(\varphi_n)$, because \eqref{sptvarphi} holds and $\text{\rm spt}(1_n)$ is compact. Since  $u$ is  $p$-superharmonic   on $X$, it follows that 
\begin{equation}
\label{test1}\int_{\text{\rm spt}(\varphi_n)}g_u^pd\mu\leq \int_{\text{\rm spt}(\varphi_n)}g_{u+\varphi_n}^pd\mu
\end{equation}
for all $n\in \mathbb{N}$. As $u+\varphi_n=\max\{M\cdot 1_n, u\}$, we have that
\begin{equation}
\label{guvarphin}g_{u+\varphi_n}^p(x)= g^p_{M\cdot 1_n}(x) \chi_{\{x\in X: M\cdot 1_n\geq u\}}(x) + g^p_u(x) \chi_{\{x\in X: u> M\cdot 1_n\}}(x)
\end{equation}
for all $x\in X$. According to  $M= \| u\|_{L^\infty(X^{n_0})}$, $1_n|_{X^{n_0}}\equiv 1$, it follows that $u(x)\leq M\cdot 1_n(x)$ and $M\cdot 1_n(x)\equiv M$ for all $x\in X^{n_0}$. Thanks to \eqref{sptvarphi}, we have that for all $x\in \text{\rm spt}(1_n)$,
\begin{equation}
\label{guvarphin1} \chi_{\{x\in X: u> M\cdot 1_n\}}(x) \leq \chi_{\{x\in \text{\rm spt}(1_n)\setminus X^{n_0}\}}(x) \text{ and }g_{M\cdot  1_n}= M\cdot   g_{1_n}\chi_{\{x\in \text{\rm spt}(1_n)\setminus  X^{n_0}\}}.
\end{equation} Substituting \eqref{guvarphin1} into \eqref{guvarphin} and combining with $\chi_{\{x\in X: M\cdot 1_n\geq u\}}\leq 1$ yields
\begin{align*}
g_{u+\varphi_n}^p(x)
&\leq M^p\cdot g^p_{1_n}(x) \chi_{\{x\in \text{\rm spt}(1_n)\setminus X^{n_0} \}}(x)+g^p_u(x)\chi_{\{x\in \text{\rm spt}(1_n)\setminus X^{n_0}\}}(x)
\end{align*}for all $x \in \text{\rm spt}(1_n)$. By \eqref{sptvarphi}, the above inequality holds for all $x\in \text{\rm spt}(\varphi_n)$. Then \eqref{test1} gives that
 \begin{equation}
 \label{gu}\int_{\text{\rm spt}(\varphi_n)}g_u^pd\mu \leq M^p\int_{\text{\rm spt}(\varphi_n)\setminus X^{n_0}}g_{1_n}^pd\mu + \int_{\text{\rm spt}(\varphi_n)\setminus X^{n_0}}g_u^pd\mu 
 \end{equation}
for all $n\in \mathbb{N}$. By $\int_{\text{\rm spt}(\varphi_n)\setminus X^{n_0}}g_u^pd\mu  <\infty$, because $u\in N^{1,p}_{\rm loc}(X)$ and $\text{\rm spt}(\varphi_n)$ is compact, subtracting $\int_{\text{\rm spt}(\varphi_n)\setminus X^{n_0}}g_u^pd\mu $ from both sides of \eqref{gu} yields
\begin{equation}\label{eq3.12}
\int_{X^{n_0}}g_u^pd\mu\leq M^p\int_{\text{\rm spt}(\varphi_n)\setminus X^{n_0}}g_{1_n}^pd\mu \leq M^p\int_{X}g_{1_n}^pd\mu 
\end{equation}
for all $n\in\mathbb{N}$. Letting $n\to \infty$, we  conclude from \eqref{0-thm1} and \eqref{eq3.12} that 
$\int_{X^{n_0}}g_u^pd\mu =0.
$ Since $n_0$ is arbitrary,  this implies that $u$ is constant,  and the claim  follows.
		\end{proof}	
	\begin{rem} \label{remark}Assume that $f>0$ is a $p$-superharmonic function on $X$. Then 
		\begin{equation}\label{caccioppoli}\int_Xf^{-p}g_f^p\varphi^pd\mu\leq \left (\frac{p}{p-1}\right )^p\int_Xg_\varphi^pd\mu 
		\end{equation}
		for all $\varphi\in N_0^{1,p}(X)$ with $0\leq \varphi\leq 1$. This inequality \eqref{caccioppoli} is often called a Caccioppoli-type inequality, see Section 3 in \cite{KM2003}.
	\end{rem}
One can give an alternate proof for Lemma \ref{lem2} via the Caccioppoli inequality \eqref{caccioppoli}. Indeed, let $u$ be an arbitrary nonnegative $p$-superharmonic function on $X$. Suppose that $X$ is $p$-parabolic. By Lemma \ref{lem1}, we have that  $\text{\rm Cap}_p(X^n)=0$ for all $n\in\mathbb{N}$. Let $n\in\mathbb{N}$ be arbitrary. Then for any $\varepsilon>0$  there exists $u_{n,\varepsilon}\in N^{1,p}_0(X)$ with $0\leq u_{n,\varepsilon}\leq 1$ and $u_{n,\varepsilon}|_{X^n}\equiv 1$ such that 
\begin{equation}
\label{capun}\int_{X}g_{u_{n,\varepsilon}}^pd\mu\leq \text{\rm Cap}_p(X^n)+\varepsilon=\varepsilon.
\end{equation}
Applying the Caccioppoli inequality \ref{caccioppoli} for $f=u+1$ with $\varphi=u_{n,\varepsilon}$, yields
\begin{equation}
\label{cacci}\int_{X}(u+1)^{-p}g_{u+1}^pu_{n,\varepsilon}^pd\mu\leq \left (\frac{p}{p-1}\right )^p\int_Xg_{u_{n,\varepsilon}}^pd\mu .
\end{equation}
Note that $g_{\log (u+1)}=(u+1)^{-1}g_{u+1}$ by \cite[Theorem 2.16 or Proposition 2.17]{BB11}. We combine this and \eqref{capun}-\eqref{cacci} with $u_{n,\varepsilon}|_{X^n}\equiv 1$ to obtain that
\[\int_{X^n}{g_{\log (u+1)}}^pd\mu\leq \left (\frac{p}{p-1}\right )^p\varepsilon.
\]Letting $\varepsilon\to 0$, this gives  $g_{\log (u+1)}=0$ on $X^n$ and hence that $u$ is constant on $X^n$. Thus $u$ is constant on $X$ since $n\in \mathbb{N}$ is arbitrary.

		{Let $x_0$ be a closest vertex of the root $0$ of a $K$-regular tree where $K\geq 2$. Then we set
			\[T_{x_0}:=\{y\in X: x_0\in[0,y]\} \text{\rm \ \ and \ \ }T_1:=[0,x_0]\cup T_{x_0}.
			\]
		For any  $n\in \mathbb{N}$, we denote
		\[E_n:=(X\setminus X^n)\cap T_1 \text{\rm \ \  and \ \ }F_n:=(X\setminus X^n)\setminus T_1.
		\]
	}
		 We define the $p$-capacity of the pair $(E_n,F_n)$, denoted $\text{\rm Cap}_p(E_n,F_n)$, by setting
		 \begin{equation}
		 \label{def-cap1}\text{\rm Cap}_p(E_n,F_n)=\inf \left \{\int_Xg_u^pd\mu: u\in N_{\rm loc}^{1,p}(X), u|_{E_n}\equiv 1, u|_{F_n}\equiv 0, 0\leq u\leq 1\right \}.	
		 \end{equation}
		The following lemma follows straightforwardly from the definition \eqref{def-cap1}  of  $p$-capacity.
{		\begin{lem}\label{lem5} {Let $X$ be a $K$-regular tree where $K\geq 2$ and let $1<p<\infty$}.  Then
			the sequence $\{\text{\rm Cap}_p(E_n,F_n)\}_{n=1}^{\infty}$ is non-increasing and there exists a constant $0<M_1<\infty$ such that
			\begin{equation}
			\label{eq-lem3.5}  \text{\rm Cap}_p(E_n,F_n)\leq M_1<\infty \text{\rm \ \  for all $n\in\mathbb N$.}
			\end{equation}			
		\end{lem}
	\begin{proof} It is clear from \eqref{def-cap1} that 
		 $\{\text{\rm Cap}_p(E_n,F_n)\}_{n=1}^{\infty}$ 		 is non-increasing and bounded from above by $\text{\rm Cap}_p(E_1,F_1)$. As the piecewise linear function $f$ with $f|_{E_1}\equiv 1$, $f|_{F_1}\equiv 0$ is admissible for computing $\text{\rm Cap}_p(E_1,F_1)$,  the claim follows since 
		 \[\text{\rm Cap}_p(E_n,F_n)\leq \text{\rm Cap}_p(E_1,F_1)\leq \int_Xg_f^pd\mu\leq  \frac{\mu(X^1)}{d(E_1,F_1)^p}<\infty \text{\rm \ \ for all $n\in\mathbb N$}.
		 \]
	\end{proof}
}
\begin{lem}\label{lem4}
{Let $X$ be a $p$-hyperbolic $K$-regular tree where $K\geq 2$ and $1<p<\infty$}. Suppose that, for each $n\in \mathbb N$, $(\text{\rm int}(X^{n}),d,\mu)$ is doubling  and supports  a $p$-Poincar\'e inequality. Then  {there exist a constant $0<M_2<\infty$ and a  sequence $\{u_n\}_{n=1}^{\infty}$ in $N_{\rm loc}^{1,p}(X)$ } with $u_n|_{E_n}\equiv 1$, $u_n|_{F_n}\equiv 0$, $0\leq u_n\leq 1$ and so that each  $u_n$ is  a nonconstant $p$-harmonic function on $\text{\rm int}(X^n)$ with
\[0< M_2\leq \int_{X^n}g_{u_n}^pd\mu =\text{\rm Cap}_p(E_n,F_n).
\]
\end{lem}
\begin{proof}
Let $n\in\mathbb N$. By the definition \eqref{def-cap1} of $\text{\rm Cap}_p(E_n,F_n)$, there exists a sequence  $\{u_{n,m}\}_{m=1}^{\infty}$ in {$N_{\rm loc}^{1,p}(X)$} with $u_{n,m}|_{E_n}\equiv 1$, $u_{n,m}|_{F_n}\equiv 0$, $0\leq u_{n,m}\leq 1$ such that 
\begin{equation}
\label{capef} \text{\rm Cap}_p(E_n,F_n)\leq \int_Xg_{u_{n,m}}^pd\mu\leq \text{\rm Cap}_p(E_n,F_n)+\frac{1}{m}.
\end{equation}
By Lemma  \ref{lem5}, there is a constant $0<M_1<\infty$ such that 
\begin{equation}
\label{eq318}{\text{\rm Cap}_p(E_n,F_n)< M_1<\infty \text{\rm \ \ for all $n\in\mathbb N$}} 
\end{equation}
and hence $\{g_{u_{n,m}} \}_{m=1}^{\infty}$ is bounded in $L^p(X)$. We have from  $\mu(X^{n+1})<\infty$, $0\leq u_{n,m}\leq 1$ that $\{u_{n,m}\}_{m=1}^{\infty}$ is bounded in $ L^p(X^{n+1})$. Then $\{u_{n,m}\}_{m=1}^{\infty}$ is bounded in $N^{1,p}(X^{n+1})$. We note that $N^{1,p}(\text{\rm int}(X^{n+1}))$ is a reflexive space since  $(\text{\rm int}(X^{n+1}),d,\mu)$ is doubling and supports a $p$-Poincar\'e  inequality, see \cite[Theorem 4.48]{C99}. Hence Theorem \ref{theorem22} gives that there is a  subsequence $\{u_{n,m_k}\}_{k=1}^\infty$ which converges weakly to some $u_n\in N^{1,p}(\text{\rm int}(X^{n+1})).$ By Mazur's Lemma  \ref{lemma23}, there is a sequence of convex combinations $f_k$ which converges to $u_n$ in $N^{1,p}(\text{\rm int}(X^{n+1}))$:
\begin{equation}
\label{fk}f_k:=\sum_{i=k}^{N_k}a_{i,k}u_{n,i}
\end{equation}
where $a_{i,k}\geq 0, \sum_{i=k}^{N_k}a_{i,k}=1$, $u_{n,i}\in \{u_{n,m_k}\}_{k=1}^\infty$.  We may assume that $f_k(x)$ converges pointwise to $u_n(x)$ as $k\to\infty$ on $\text{\rm int}(X^{n+1})$. {It is  easy  to see  that  $u_n|_{\text{\rm int}(X^{n+1})\cap  E_n}\equiv 1$ and $u_n|_{\text{\rm int}(X^{n+1})\cap F_n}\equiv 0$ from \eqref{fk}  since  $u_{n,i}|_{\text{\rm int}(X^{n+1})\cap E_n}\equiv 1$, $u_{n,i}|_{\text{\rm int}(X^{n+1})\cap F_n}\equiv 0$ for all $i$. We extend $u_n$ by setting $u_n|_{E_n\setminus\text{\rm int}(X^{n+1})}\equiv 1$, $u_n|_{F_n\setminus\text{\rm int}(X^{n+1})}\equiv 0$}. Then
\begin{equation}
{\label{un} u_n\in L_{\rm loc}^p(X) \text{ with  }u_n|_{E_n}\equiv 1, u_n|_{F_n}\equiv 0, 0\leq u_n\leq  1.}
\end{equation}
Next, 
we have by the convexity  of the function $t\mapsto t^p$  that 
\[
\int_{X^n}g_{f_k}^pd\mu\leq \sum_{i=k}^{N_k}a_{i,k}\int_{{X^n}}g_{u_{n,i}}^pd\mu.
\]
By the triangle inequality, this gives
\begin{align}
\left (\int_{X^n}g_{u_n}^pd\mu \right )^{\frac{1}{p}} &\leq \left (\int_{X^n}g_{f_k-u_n}^pd\mu \right )^{\frac{1}{p}} + \left (\int_{X^n}g_{f_k}^pd\mu \right )^{\frac{1}{p}} \notag \\ 
&\leq \left (\int_{X^n}g_{f_k-u_n}^pd\mu \right )^{\frac{1}{p}}  + \left (\sum_{i=k}^{N_{k}} a_{i,k}\int_{X^n}g_{u_{n,i}}^pd\mu \right )^{\frac{1}{p}}.\label{eq39}
\end{align} 
According to  \eqref{capef} and $\sum_{i=k}^{N_k}a_{i,k}=1$,
\begin{align}
\left (\sum_{i=k}^{N_{k}} a_{i,k}\int_{X^n}g_{u_{n,i}}^pd\mu \right )^{\frac{1}{p}}&\leq  \left (\sum_{i=k}^{N_{k}} a_{i,k} \text{\rm Cap}_p(E_n,F_n) + \sum_{i=k}^{N_{k}} a_{i,k} \frac{1}{i}\right )^{\frac{1}{p}} \notag \\
&\leq   \left (\text{\rm Cap}_p(E_n,F_n) + \frac{1}{k}\right )^{\frac{1}{p}}.\label{eq310}
\end{align}
Substituting \eqref{eq310} into \eqref{eq39} and combining with  $f_k\to u_n$ in $N^{1,p}(\text{\rm int}(X^{n+1}))$ as $k\to\infty$, we obtain that 
\begin{align}
\left (\int_{X^n}g_{u_n}^pd\mu \right )^{\frac{1}{p}} 
&\leq  \left (\int_{X^n}g_{f_k-u_n}^pd\mu \right )^{\frac{1}{p}}  + \left ( \text{\rm Cap}_p(E_n,F_n) +  \frac{1}{k}\right )^{\frac{1}{p}} 
\to \text{\rm Cap}_p(E_n,F_n) ^{\frac{1}{p}}\notag 
\end{align} as $k\to \infty$. Then the above estimate gives via \eqref{eq318} and {$u_n|_{F_n}\equiv0, u_n|_{E_n}\equiv 1$} that $g_{u_n}\in L^{p}(X)$. Combining this with \eqref{un} yields {$u_n\in N_{\rm loc}^{1,p}(X)$} with $u_n|_{E_n}\equiv 1$, $u_n|_{F_n}\equiv 0$, $0\leq u_n\leq 1$ and hence $u_n$ is admissible for computing the capacity $\text{\rm Cap}_p(E_n,F_n) $.  It follows from this and the above estimate that
\begin{equation}
\label{capef=u} \int_{X^n}g_{u_n}^pd\mu=\text{\rm Cap}_p(E_n,F_n) \text{\rm \ \ for all $n\in \mathbb N.$}
\end{equation} 
We conclude from \eqref{eq318},\eqref{un},\eqref{capef=u}
that {there are a constant $0<M_1<\infty$ and a function $u_n\in N_{\rm loc}^{1,p}(X)$ with $u_n|_{E_n}\equiv 1, u_n|_{F_n}\equiv 0$, and  $0\leq u_n\leq 1$ such that}
\begin{equation}
\label{finite}{ \int_{X^n}g_{u_n}^pd\mu=\text{\rm Cap}_p(E_n,F_n)\leq M_1<\infty \text{\rm \ \ for all $n\in\mathbb N$}.}
\end{equation}
{
We now prove that there is a constant $0<M_2<\infty$ such that 
\begin{equation}
\label{constant-m} 0<M_2\leq \int_{X^n}g_{u_n}^pd\mu=\text{\rm Cap}_p(E_n,F_n) \text{\rm \ \ for all $n\in\mathbb N$}.
\end{equation}
Let $n\in\mathbb N$ and $u_n$ be as in \eqref{finite}. If $u_n(0)<1/2$, we define $v:=\max\{0,\min\{1,2(1-u_n)\}\}$. Then $v\in N^{1,p}_{\rm loc}(X)$ with $v(0)=1$, $v|_{E_n}\equiv 0$, and hence
\[\text{\rm Cap}_p(\{0\},E_n)\leq \int_X g_v^pd\mu \leq 2^p\int_{X^n}g_{u_n}^pd\mu= 2^p\text{\rm Cap}_p(E_n,F_n).
\] If $u_n(0)\geq 1/2$, we define $w:=\max\{0,\min\{1,2u_n\}\}$. Then $w\in N^{1,p}_{\rm loc}(X)$ with $w(0)=1$, $w|_{F_n}\equiv 0$, and hence
\[\text{\rm Cap}_p(\{0\},F_n)\leq \int_X g_w^pd\mu \leq 2^p\int_{X^n}g_{u_n}^pd\mu= 2^p\text{\rm Cap}_p(E_n,F_n).
\]Combining the above estimates for each $n\in\mathbb N$, we obtain that 
	\[ \inf_{n\in\mathbb N}\min\{\text{\rm Cap}_p(\{0\},E_n), \text{\rm Cap}_p(\{0\},F_n)\}\leq 2^p \int_{X^n}g_{u_n}^pd\mu=2^p\text{\rm Cap}_p(E_n,F_n) \text{\rm \ \ for all $n\in\mathbb N$}.
	\]
To obtain \eqref{constant-m}, we show that 
\begin{equation}
\label{positive-m}M_2:=\inf_{n\in\mathbb N}\min\{\text{\rm Cap}_p(\{0\},E_n), \text{\rm Cap}_p(\{0\},F_n)\}>0.
\end{equation}
Applying Lemma \ref{lem1} and Lemma \ref{prop} for the subtrees $T_1$ and $X\setminus T_1$, we obtain that 
\[\text{\rm $T_1$ is $p$-hyperbolic if and only if $R_p^{T_1}:=\frac{1}{K}R_p<\infty$}
\]
and
\[\text{\rm $X\setminus T_1$ is $p$-hyperbolic if and only if $R_p^{X\setminus T_1}:=\frac{K-1}{K}R_p<\infty$}.
\]
Notice that if $X$ is $p$-hyperbolic, then both $T_1$ and $X\setminus T_1$ are $p$-hyperbolic since $R_p<\infty$ implies finiteness of $R_p^{T_1}$ and $R_p^{X\setminus T_1}$. By Lemma \ref{prop} with the $p$-hyperbolicity of $T_1$ and $X\setminus T_1$, it follows that 
\[ \text{\rm Cap}_p(\{0\},E_n) \geq \text{\rm Cap}_p^{T_1}(\{0\}):=\inf\left \{ \int_{T_1}g_u^pd\mu: u\in N^{1,p}_0(T_1), u(0)=1\right \}>0
\]
and
\[ \text{\rm Cap}_p(\{0\},F_n) \geq \text{\rm Cap}_p^{X\setminus T_1}(\{0\}):=\inf\left \{ \int_{X\setminus T_1}g_u^pd\mu: u\in N^{1,p}_0(X\setminus T_1), u(0)=1\right \}>0
\]
for each $n\in\mathbb N$. Hence $\eqref{positive-m}$ holds.
}

Finally, we only need to show that $u_n$ is a $p$-harmonic function on $\text{\rm int}(X^n)$. Let $\varphi$ be an arbitrary element of $N^{1,p}(\text{\rm int}(X^n))$ with compact support $\text{\rm spt}(\varphi)\subset \text{\rm int}(X^n)$. By choosing $v=\max\{0,\min\{1,u_n+\varphi\} \}$ we have  that $v\in N^{1,p}(\text{\rm int}(X^n))$ with $v|_{E_n}\equiv 1, v|_{F_n}\equiv 0, 0\leq v\leq 1$ because $\text{\rm spt}(\varphi)\subset \text{\rm int}(X^n)$ and because \eqref{un} holds. It follows from the definition \eqref{def-cap1} of $\text{\rm Cap}_p(E_n,F_n)$ that 
\[
\text{\rm Cap}_p(E_n,F_n)\leq \int_{X^n}g_v^pd\mu \leq \int_{X^n}g_{u_n+\varphi}^pd\mu .
\]Combining this with \eqref{capef=u}, we obtain that 
\[\int_{X^n}g_{u_n}^pd\mu\leq \int_{X^n}g_{u_n+\varphi}^pd\mu 
\]for all $\varphi\in N^{1,p}(\text{\rm int}(X^n))$ with compact support $\text{\rm spt}(\varphi)\subset \text{\rm int}(X^n)$. Hence $u_n$ is  $p$-harmonic  on $\text{\rm int}(X^n)$, and the claim follows.
\end{proof}
\begin{lem}\label{lem6}
{Let $X$ be a $p$-hyperbolic $K$-regular tree where $K\geq 2$ and $1<p<\infty$.} Suppose that $(\text{\rm int}(X^{n}),d,\mu)$ is doubling  and supports  a $p$-Poincar\'e inequality   for  each $n\in\mathbb{N}$. Then there exists a nonconstant nonnegative bounded $p$-harmonic function $u$ on $X$ with $0<\int_Xg_u^pd\mu <\infty$.
\end{lem}
\begin{proof}
	We first  prove  that there exists  a nonnegative  bounded $p$-harmonic function  $u$ on $X$. Towards  this, 
 Lemma \ref{lem5} and Lemma \ref{lem4} give  that { there exist  constants $0<M_2\leq M_1<\infty$ and a sequence  $\{u_n\}_{n=1}^{\infty}$ in $N_{\rm loc}^{1,p}(X)$ with} $u_n|_{E_n}\equiv 1$, $u_n|_{F_n}\equiv 0$, $0\leq u_n\leq 1$  such that $u_n$ is a nonconstant $p$-harmonic function on $\text{\rm int}(X^n)$ and
\begin{equation}
\label{36}0<{M_2}\leq \int_{X^n}g_{u_n}^pd\mu=\text{\rm Cap}_p(E_n,F_n) \leq M_1<\infty \text{\rm \ \ for all $n\in\mathbb N$.}
\end{equation}
 Let $n_0\in\mathbb N$ be arbitrary.  It follows from the local H\"older continuity of $p$-harmonic functions (see Theorem \ref{theorem26}) that  $\{u_n\}_{n\geq n_0}$ is  equibounded and locally equicontinuous on $\text{\rm int}(X^{n_0+1})$. By the Arzel\`{a}-Ascoli theorem, there exists a subsequence, still denoted $\{u_n\}_{n\geq n_0}$,  that converges to $u$  uniformly on $X^{n_0}$  as $n\to\infty$. Since $n_0$ is arbitrary, by uniqueness of locally uniform convergence, we may assume that $u_n$ converges to $u$ locally uniformly in $X$ as  $n\to\infty$. Then $\{u_i\}_{i\geq n}$ is a sequence of $p$-harmonic functions on $\text{\rm int}(X^n)$ which converge  uniformly to $u$  in $\text{\rm int}(X^{n})$ for each $n$, and hence by Theorem \ref{theorem25} and Theorem \ref{theorem24} we obtain that $u$  is a $p$-harmonic function on $X$. Thus $u$ is a nonnegative bounded $p$-harmonic function on $X$ since  $0\leq u_n\leq 1$.
 
 We  next  show  that 
 \[0<\int_Xg_u^pd\mu<\infty.
 \]
 It follows from \eqref{36}  that $\{g_{u_n}\}_{n=1}^{\infty}$ is a bounded sequence in the reflexive space $L^{p}(X)$, and hence Theorem \ref{theorem22} and Mazur's Lemma  \ref{lemma23} give that there exists $g\in L^{p}(X)$ and a convex combination sequence $\bar g_n=\sum_{i=n}^{N_n}a_{i,n}g_{u_i}$ with $a_{i,n}\geq 0, \sum_{i=n}^{N_n}a_{i,n}=1$ such that $\bar g_n\to g$ in $L^p(X)$ as $n\to \infty$. By Fuglede's Lemma  \ref{lemma21}, we obtain that there is a subsequence, still denoted $\bar g_n$, such that 
 \[\lim_{n\to \infty}\int_{[x,y]}\bar g_nds=\int_{[x,y]}gds\]for every curve $[x,y]$. Note that $\bar g_n=\sum_{i=n}^{N_n}a_{i,n}g_{u_i}$ is an upper gradient of $\bar u_n= \sum_{i=n}^{N_n}a_{i,n}u_i$ and $\bar u_n$ converges to $u$ locally uniformly in $X$ as $n\to \infty$. Hence 
 \[|u(x)-u(y)|=\lim_{n\to\infty}|\bar u_n(x)-\bar u_n(y)|\leq \lim_{n\to\infty}\int_{[x,y]}\bar g_nds=\int_{[x,y]}gds
 \] for every curve $[x,y]$. Then $g$ is an upper gradient of $u$ and hence $g_u\leq g$ a.e.. Combining this with $\bar g_n\to g$ in $L^p(X)$ as $n\to\infty$ and with the convexity of the function $t\mapsto t^p$ yields
 \begin{equation}\label{37}
 \int_{X}g_u^pd\mu \leq \int_{X}g^pd\mu =\lim_{n\to\infty}\int_X \bar g_n^pd\mu  \leq \lim_{n\to\infty} \sum_{i=n}^{N_n}a_{i,n}\int_Xg_{u_{i}}^pd\mu.
 \end{equation}
 Notice that $\{\text{\rm Cap}_{p}(E_n,F_n)\}_{n=1}^{\infty}$ is a nonincreasing sequence by Lemma \ref{lem5}. Hence we have by 
 {$u_i|_{F_i}\equiv0, u_i|_{E_i}\equiv 1$} and \eqref{36} that 
 \begin{align*}
 \lim_{n\to\infty} \sum_{i=n}^{N_n}a_{i,n}\int_Xg_{u_{i}}^pd\mu&= \lim_{n\to\infty}\sum_{i=n}^{N_n}a_{i,n}\int_{X^i}g_{u_i}^pd \mu  = \lim_{n\to\infty}\sum_{i=n}^{N_n}a_{i,n}\text{\rm Cap}_p(E_{i},F_i)\\
 & \leq \lim_{n\to\infty}\sum_{i=n}^{N_n}a_{i,n}\text{\rm Cap}_p(E_{n},F_{n})= \lim_{n\to\infty}\text{\rm Cap}_p(E_{n},F_{n})<\infty.
 \end{align*}
 Substituting the above estimate into \eqref{37} yields
  $\int_Xg_u^pd \mu<\infty$. 
  
  It remains to show that $\int_Xg_u^pd\mu>0$. 
  The preceding being understood, we argue by contradiction and assume that $\int_Xg_u^pd\mu=0$. Then $u$ is constant on $X$.
Recall that $\bar g_n=\sum_{i=n}^{N_n}a_{i,n}g_{u_i}$ with $a_{i,n}\geq 0, \sum_{i=n}^{N_n}a_{i,n}=1$ are such that $\bar g_n\to g$ in $L^p(X)$ as $n\to \infty$
 and 
\begin{equation}
\label{gds}\lim_{n\to\infty}\int_{[x,y]}\bar g_{n}ds=\int_{[x,y]}gds
\end{equation} for every curve $[x,y]$.  
 Moreover, $\bar{g}_n$ is an upper gradient of $\bar{u}_n=\sum_{i=n}^{N_n}a_{i,n}{u_i}$ and  $\bar{u}_n$ is admissible for computing the capacity $\text{\rm Cap}_p(E_{N_n},F_{N_n})$, and so 
\[\int_X\bar{g}_n^pd \mu \geq \int_{X}g_{\bar{u}_n}^pd\mu\geq \text{\rm Cap}_p(E_{N_n},F_{N_n}).
\]Combining this with  \eqref{36} and using $\bar{g}_n\to g$ in $L^p(X)$ as $n\to\infty$ yields
\begin{equation}
\label{gdmu}\int_Xg^pd\mu >0.
\end{equation}
Let $\varepsilon>0$ be arbitrary, and let $[x,y]$ be an arbitrary edge  in $X$. 
Since  $u_n$ converges to $u$ locally uniformly in $X$, there exist positive constants $N, r$ only depending on $x, y$ such that for all $n\geq N$,
\[\sup_{t\in B(x,r)}|u_n(t)-u(t)|<\varepsilon, \sup_{t\in B(y,r)}|u_n(t)-u(t)|<\varepsilon
\] where $B(x,r), B(y,r)$ are balls with centers $x, y$ and  radius $r$, respectively. Since $u$ is constant, the above estimates  yield that  for all $n\geq N$, 
\[
|{u}_n(x)-{u}_n(y)|\leq |u_n(x)-u(x)|+|u(y)-u_n(y)|<2\varepsilon.
\] 
By Section \ref{sec22}, $u_n$ is absolutely continuous on $[x,y]$ and $g_{u_n}(z)=|u_n'(z)|/\lambda(z)$ for $z\in [x,y]$. By the strong maximum principle (see for instance \cite[Corollary 6.5]{KS01}), $u_n$ is a monotone function on $[x,y]$. Hence
\[|u_n(x)-u_n(y)|=\int_{[x,y]}|u_n'(z)|dz.
\]
Since  the minimal upper gradient $g_{u_n}$ of $u_n$ satisfies $g_{u_n}(x)=|u'_n(x)|/\lambda(z)$, we conclude that 
\[\int_{[x,y]}g_{u_n}ds\leq 2\varepsilon
\] for all $n\geq N$.
By $\eqref{gds}$, it follows that
\[\int_{[x,y]}gds=\lim_{n\to \infty}\int_{[x,y]}\bar g_nds=\lim_{n\to \infty}\sum_{i=n}^{N_n}a_{i,n}\int_{[x,y]}g_{u_i}ds<2\varepsilon.
\] Letting $\varepsilon\to 0$, we obtain that $g=0$ a.e. on $[x,y]$. As $[x,y]$ is arbitrary, we conclude that  $g=0$ a.e.  which contradicts  \eqref{gdmu}. This completes the proof.
\end{proof}
\begin{proof}[Proof of Theorem \ref{main theorem}]
	$X$ is $p$-parabolic $\Leftrightarrow$ $(2.)$ is given by Lemma \ref{lem1}. 
	
	$(1.)\Leftrightarrow (2.) \Leftrightarrow (3.)$ is given by Lemma \ref{prop}.
\end{proof}
\begin{proof}[Proof of Theorem \ref{main theorem2}] 
	$X$ is $p$-parabolic  $ \Rightarrow (1.)$ is given by Lemma \ref{lem2}.
	
	$(1.)\Rightarrow (2.)$ is trivial.	
	
	$(2.) \Rightarrow (3.)$: Let $u$ be a bounded $p$-harmonic function on $X$. Then there exists a constant $C>0$ such that $u+C$ is a nonnegative $p$-harmonic function on $X$. Hence $u+C$ is constant by the assumption and so $u$ is constant.
	
	$(3.)\Rightarrow (4.)$ is trivial.
	
	$(4.)\Rightarrow$ $X$ is $p$-parabolic is given by Lemma \ref{lem6}. 
\end{proof}
The following example shows that Theorem \ref{main theorem} is not true for general metrics and measures.
\begin{example}\label{example}Let $1<p<\infty$. There exists a $p$-hyperbolic $K$-regular tree $X$ with a distance and a ``{non-radial}'' measure such that $R_p=\infty$.
\end{example}
Let us begin with some notation. For simplicity,  let $X$ be a dyadic tree (which means $K=2$). Then the root $0$ of our tree has two closest vertices, denoted $v_1$ and $v_2$. We denote
\[T_1=[0,v_1]\cup \{x\in X: v_1 \in[0,x]\} \text{ and }T_2=[0,v_2]\cup \{x\in X: v_2\in[0,x]\}.
\]Note that the union of $T_1$ and $T_2$ is our tree. Suppose $\lambda_i, \mu_i:[0, \infty)\rightarrow (0, \infty)$ satisfy $\lambda_i, \mu_i\in L^1_{\rm loc}([0, \infty))$, for $i=1,2$. We introduce a measure $\mu$ and a metric $d$ via $ds$ by setting
\begin{equation*}
d\mu(x)=\mu_i(|x|)\,d|x|,\ \  ds(x)=\lambda_i(|x|)\, d|x|,
\end{equation*}
for all $x\in T_i$, for $i=1,2$.
To obtain what we desire, we choose $\lambda_1\equiv \mu_1\equiv 1$ and $\lambda_2\equiv 1, \mu_2(x)=2^{-j(x)}$. Define a metric $d$ and a measure $\mu$ as above. Then
	\[R_p|_{T_1}:=\frac{1}{2}\int_0^\infty \lambda_1^{\frac{p}{p-1}}(t)\mu_1^{\frac{1}{1-p}}(t)2^{\frac{j(t)}{1-p}}dt<\infty, \ \ \ 
	R_p|_{T_2}:=\frac{1}{2}\int_0^\infty \lambda_2^{\frac{p}{p-1}}(t)\mu_2^{\frac{1}{1-p}}(t)2^{\frac{j(t)}{1-p}}dt=\infty,
	\] and $R_p=R_p|_{T_1}+R_p|_{T_2}=\infty$. By Theorem \ref{main theorem} for the subtree $T_1$ with $R_p|_{T_1}<\infty$, we obtain that $T_1$ is $p$-hyperbolic. Hence there exists a compact set $O$ in $T_1$ such that $\text{\rm Cap}_p^{T_1}(O)>0$, where 
	\[\text{\rm Cap}_p^{T_1}(O):=\inf\left \{\int_{T_1} g_u^pd\mu: u|_O\equiv 1, u\in N^{1,p}_0(T_1)\right \}.
	\]
	Let $O$ be such a set. Then $O$ is bounded in $T_1$ by Corollary \ref{cor2.2}. Let $u\in N^{1,p}_0(X)$ be an arbitrary function with $u|_O\equiv 1$. It follows from $u\in N^{1,p}_0(X)$ that there exists a sequence $u_n\in N^{1,p}(X)$ with compact support $\text{\rm spt}(u_n)\subset X$ such that $u_n\to u$ in $N^{1,p}(X)$ as $n\to \infty$. By Corollary \ref{cor2.2}, we may assume that $\text{\rm spt}(u_n)\subset X^n$ for each $n$. Hence $\text{\rm spt}(u_n)\cap T_1 \subset X^n\cap T_1$, and so $\text{\rm spt}(u_n)\cap T_1$ is compact in $T_1$ because $\text{\rm spt}(u_n)\cap T_1$ is a closed set in $T_1$ and $X^n\cap T_1$ is compact in $T_1$. Then for each $n$, $u_n\in N^{1,p}(T_1)$ with compact support $\text{\rm spt}(u_n)$ and $u_n\to u$ in $N^{1,p}(T_1)$ as $n\to \infty$, and hence $u\in N^{1,p}_0(T_1)$ with $u|_O\equiv 1$.
	 Thus $u$ is admissible for computing $\text{\rm Cap}_p^{T_1}(O)$ and so
	\[\int_{X}g_u^pd\mu\geq \int_{T_1}g_u^pd\mu\geq \text{\rm Cap}_p^{T_1}(O).
	\]
	Since $u\in N^{1,p}_0(X)$ with $u|_O\equiv 1$ is arbitrary, the above estimate gives 
	\[\text{\rm Cap}_p(O)\geq \text{\rm Cap}_p^{T_1}(O).
	\]
	Combining this with $\text{\rm Cap}_p^{T_1}(O)>0$, we have $\text{\rm Cap}_p(O)>0$. Thus $X$ is $p$-hyperbolic.

\begin{example}
	\label{example2}Let $1<p<\infty$. There exist both  $p$-hyperbolic and $p$-parabolic $K$-regular trees $(X,d,\mu)$ that are  doubling and  support  a  $p$-Poincar\'e inequality. 
\end{example}
We begin with the $p$-hyperbolic case.
Let $\mu(t)=e^{-\beta j(t)}$ and $\lambda(t)=e^{-\varepsilon j(t)}$ with $\varepsilon, \beta>0$ and $\log K<\beta<\log K+\varepsilon p$. It is obvious that $\mu(X)<\infty$ and $R_p<\infty$. More precisely, since $\log K<\beta<\log K+\varepsilon p$ we have that 
\[\mu(X)=\int_0^\infty \mu(t)K^{j(t)}dt=\int_0^\infty e^{-(\beta -\log K)j(t)}dt<\infty
\]
and
\[R_p=\int_0^\infty \lambda(t)^{\frac{p}{p-1}}\mu(t)^{\frac{1}{1-p}}K^{\frac{j(t)}{1-p}}dt=\int_0^\infty e^{\frac{(\beta-\log K-\varepsilon p)j(t)}{p-1}}dt<\infty.
\]As $R_p<\infty$, by Theorem \ref{main theorem}, it follows that $(X,d,\mu)$ is a noncomplete $p$-hyperbolic metric measure space. By \cite[Section 3 and Section 4]{BBGS} or \cite[Section 2]{PW19} for $\log K<\beta$, we obtain that $(X,d,\mu)$ is doubling and supports a $1$-Poincar\'e inequality.

For the $p$-parabolic case, let $\mu(t)=e^{-\beta j(t)}$ and $\lambda(t)=e^{-\varepsilon j(t)}$ with $\varepsilon,\beta>0$ and $\beta=\log K+\varepsilon p$. It is easy to see that 
$(X,d,\mu)$ is a noncomplete $p$-parabolic $K$-regular tree that is doubling and supports a $1$-Poincar\'e inequality.
\section*{Acknowledgement}
The author thanks his advisor Professor Pekka Koskela for helpful discussions. The author also would like to thank Zhuang Wang for reading the manuscript and giving comments that helped to improve the paper. 
The author is extremely grateful to Professor Jana Bj\"orn and the referees for their comments, suggestions, and corrections.
	
	\end{document}